\documentclass[10pt]{amsart}
\usepackage[abbrev, alphabetic, y2k]{amsrefs}
\usepackage{mathptmx}
\usepackage{color}
\usepackage{amssymb}
\usepackage{tikz-cd}

\newif\ifdebug

\debugtrue

\theoremstyle{plain}
\newtheorem{theo}{Theorem}[section]
\newtheorem{lemm}[theo]{Lemma}

\newtheorem{prop}[theo]{Proposition}
\newtheorem{conj}[theo]{Conjecture}

\theoremstyle{remark}
\newtheorem{rema}[theo]{Remark}

\theoremstyle{definition}

\def\C{\mathbb C}
\def\Z{\mathbb Z}

\def\co{\thinspace \colon}

  \makeatletter
    
    \@addtoreset{equation}{section}
  \makeatother

\begin{document}
\title[Strong cohomological rigidity of Hirzebruch surface bundles]{Strong cohomological rigidity of Hirzebruch surface bundles in Bott towers
}

\author[H.~Ishida]{Hiroaki Ishida}
\address{Department of Mathematics and Computer Science, Graduate School of Science and Engineering, Kagoshima University}
\email{ishida@sci.kagoshima-u.ac.jp}

\date{\today}
\thanks{This work is supported by JSPS KAKENHI Grant Number JP20K03592}


\begin{abstract}
	We show the strong cohomological rigidity of Hirzebruch surface bundles over Bott manifolds. 
	As a corollary, we have that the strong cohomological rigidity conjecture is true for Bott manifolds of dimension $8$. 
\end{abstract}

\maketitle
\section{Introduction}\label{sec:intro}
	A \emph{Bott tower} of height $n$ is an iterated $\C P^1$-bundle
	\begin{equation}\label{eq:Botttower}
		\begin{tikzcd}
			B_{n} \arrow[r, "\pi_{n}"] & B_{n-1} \arrow[r, "\pi_{n-1}"]  &  \cdots \ar[r, "\pi_2"] & B_1 \ar[r, "\pi_1"] & B_0 = \{\text{a point}\}
		\end{tikzcd}
	\end{equation}
	where each fibration $\pi_j \co B_j \to B_{j-1}$ is the projectivization of a Whitney sum of two complex line bundles over $B_{j-1}$. 
	Each $B_j$ is called a \emph{Bott manifold}. The notion of the Bott tower is introduced in \cite{GK1994}.
By definition, $B_j$ is a closed manifold of dimension $2j$ and $B_1$ is nothing but the projective line $\C P^1$. Bott manifolds of dimension $4$ are known as \emph{Hirzebruch surfaces} and introduced in \cite{Hirzebruch1951}. 
	
	By composing the first $k$ fibrations $\pi_n, \dots, \pi_{n-k+1}$, we have a fiber bundle $B_n \to B_{n-k}$ whose fibers are Bott manifold of dimension $2k$. Moreover the cohomology $H^*(B_n)$ of $B_n$ has a structure of $H^*(B_{n-k})$-algebra. In this paper, we clarify the relation between the topology of Hirzebruch surface bundles over a Bott manifold $B_n$ and the structure of the cohomology as $H^*(B_n)$-algebras. 
	The main theorem of this paper is the following.
	\begin{theo}\label{theo:maintheo}
		Let 
		\begin{equation*}
			\begin{tikzcd}
				B_{n+2} \arrow[r, "\pi_{n+2}"] & B_{n+1} \arrow[r, "\pi_{n+1}"]  & B_n \arrow[r, "\pi_n"] & \cdots \ar[r, "\pi_2"] & B_1 \ar[r, "\pi_1"] & B_0 = \{\text{a point}\}
			\end{tikzcd}
		\end{equation*}
		and 
		\begin{equation*}
		\begin{tikzcd}
			B_{n+2}' \arrow[r, "\pi_{n+2}'"] & B_{n+1}' \arrow[r, "\pi_{n+1}'"]  & B_n \arrow[r, "\pi_n"] & \cdots \ar[r, "\pi_2"] & B_1 \ar[r, "\pi_1"] & B_0 = \{\text{a point}\}
		\end{tikzcd}
	\end{equation*}
		be Bott towers of height $n+2$. Let $\varphi \co H^*(B_{n+2}) \to H^*(B_{n+2}')$ be an isomorphism as $H^*(B_n)$-algebras. Then, there exists a bundle isomorphism $f \co B_{n+2}' \to B_{n+2}$ over $B_n$ such that $f^* = \varphi$. 
	\end{theo}
	This study is motivated by the \emph{strong cohomological rigidity conjecure} posed in \cite{Choi2015}. 
	\begin{conj}[Strong cohomological rigidity conjecture for Bott manifolds]\label{conj:scrc}
		Any graded ring isomorphism between the cohomology rings of two Bott manifolds is induced by a diffeomorphism.
	\end{conj}
	So far, no counterexamples are known to Conjecture \ref{conj:scrc} and some partial affirmative results are known. Conjecture \ref{conj:scrc} is known to be true for Bott manifolds of dimension up to $6$ (see \cite[Theorem A]{Choi2015}). Conjecture \ref{conj:scrc} is also true if we restrict Bott manifolds to $\mathbb{Q}$-trivial (see \cite[Corollary 5.1]{CM2012}) or $\Z/2\Z$-trivial (\cite[Theorem 1.3]{CMM2015}) ones. A Bott manifold $B_n$ is said to be $\mathbb{Q}$-trivial (respectively, $\Z/2\Z$-trivial) if $H^*(B_n)\otimes \mathbb{Q} \cong H^*((\C P^1)^n)\otimes \mathbb{Q}$ (respectively, $H^*(B_n)\otimes \Z/2\Z \cong H^*((\C P^1)^n)\otimes \Z/2\Z$). In \cite[Problem 3.5]{Choi2015}, it is pointed out that Theorem \ref{theo:maintheo} implies that Conjecture \ref{conj:scrc} is also true for Bott manifolds of dimension $8$. 
	
	This paper is organized as follows. In Section \ref{sec:preliminaries}, we recall some preliminary facts about Bott manifolds that we need. In Section \ref{sec:Hirzebruch}, we study the topology of Hirzebruch surfaces. Especially, we show an $S^1$-equivariant version of the cohomological rigidity of a Hirzebruch surface. Sections \ref{sec:auto}, \ref{sec:realizing}, \ref{sec:rigidity} are devoted to prove Theorem \ref{theo:maintheo}. In Section \ref{sec:auto}, we study the automorphism of $H^*(B_{n+2})$ as an $H^*(B_n)$-algebra. In Section \ref{sec:realizing}, we show that any automorphism of $H^*(B_{n+2})$ as an $H^*(B_n)$-algebra is induced by a bundle automorphism. In Section \ref{sec:rigidity}, we show that isomorphism classes of Hirzebruch surface bundles over a Bott manifold $B_n$ are distinguished by their cohomologies as $H^*(B_n)$-algebras. These results obtained in Sections \ref{sec:realizing} and \ref{sec:rigidity} imply Theorem \ref{theo:maintheo} immediately.

	Throughout this paper, all cohomologies are taken with $\Z$-coefficient. 
	
	\noindent {\bf Acknowledgement.} The author is grateful to the anonymous referee for the invaluable comments and useful suggestions on improving the text.
	
\section{Preliminaries}\label{sec:preliminaries}
	In this section we recall elementary facts about Bott manifolds for later use. 
	\begin{lemm}[{\cite[Lemma 2.1]{CMS2010}}]\label{lemm:tensor}
		Let $B$ be a smooth manifold. Let $L$ be a complex line bundle over $B$ and $V$ a complex vector bundle over $B$. Then, the projectivizations of $V$ and $L \otimes V$ are  isomorphic as bundles over $B$. 
	\end{lemm} 
	Thanks to Lemma \ref{lemm:tensor} we may assume that one of line bundles over $B_j$ is the product line bundle $\underline{\C}$ in \eqref{eq:Botttower} without loss of generality. Throughout this paper, we assume that each fibration $B_j \to B_{j-1}$ is a projectivization of $P(\underline{\C} \oplus \xi_j) \to B_{j-1}$, where $\underline{\C}$ is a product line bundle and $\xi_j$ is a complex line bundle over $B_{j-1}$. 
	
	Let $B$ be a smooth manifold and $V$ a complex vector bundle over $B$. The tautological line bundle $\gamma$ of $P(V)$ is a line bundle over $P(V)$ given by 
	\begin{equation*}
		\gamma = \{ (\ell, v) \in P(V) \times V \mid \ell \ni v\},
	\end{equation*}
	here we think of $P(V)$ as a set of all lines passing $0$ in $V$. The first Chern classes of tautological bundles play a role to describe the cohomology ring of Bott manifolds. 
	Lemma \ref{lemm:Leray-Hirsch} and Theorem \ref{theo:cohomology} below are well known facts. For leader's convenience, we give brief proofs. 
	\begin{lemm}\label{lemm:Leray-Hirsch}
		Let $B$ be a smooth manifold. Let $\xi$ be a complex line bundle over $B$. Then $H^*(P(\underline{\C}\oplus \xi))$ is isomorphic to $H^*(B)[X]/(X^2-c_1(\xi)X)$ as an $H^*(B)$-algebra, where $\deg X = 2$. 
	\end{lemm}
	\begin{proof}
		Let $\gamma$ be the tautological line bundle of $P(\underline{\C}\oplus \xi)$. Let $x \in B$ and $(\underline{\C} \oplus \xi)_x$ the fiber of $\underline{\C} \oplus \xi$ at $x$. Then, the restriction of $\gamma$ to $P((\underline{\C} \oplus \xi)_x)$ is nothing but the tautological line bundle over the projective line $P((\underline{\C} \oplus \xi)_x)$. Since $H^*(P((\underline{\C} \oplus \xi)_x))$ is generated by $1$ and the first Chern class of the tautological line bundle as a $\Z$-module, Leray-Hirsch Theorem (see \cite{Hatcher2002} for example) yields that $H^*(P(\underline{\C}\oplus \xi))$ is generated by $1$ and $c_1(\gamma)$ as an $H^*(B)$-module.
		
		Let $\pi \co P(\underline{\C} \oplus \xi) \to B$ be the projection. We show that $c_1(\gamma)(-c_1(\gamma)+\pi^*c_1(\xi)) =0$. Let $\gamma^\perp$ be the orthogonal complement of $\gamma$ in $\underline{\C} \oplus \xi$ for some hermitian metric. Since $\gamma \oplus \gamma^\perp = \pi^*(\underline{\C} \oplus \xi)$, by comparing Chern classes we have that $c_1(\gamma^{\perp}) = -c_1(\gamma)+\pi^*c_1(\xi)$ and $c_1(\gamma)c_1(\gamma^\perp) = 0$. Thus we have $c_1(\gamma)(-c_1(\gamma)+\pi^*c_1(\xi)) =0$. 
		
		The computation above shows that the homomorphism $H^*(B)[X]/(X^2-c_1(\xi)X) \to H^*(P(\underline{\C}\oplus \xi))$ given by $X \mapsto c_1(\gamma)$ is an isomorphism as $H^*(B)$-algebras. 
	\end{proof}
	Let 
	\begin{equation*}
		\begin{tikzcd}
			B_\bullet \co  B_n \arrow[r, "\pi_n"] & B_{n-1} \ar[r, "\pi_{n-1}"] & \cdots \ar[r, "\pi_2"] & B_1 \ar[r, "\pi_1"] & B_0 = \{\text{a point}\}
		\end{tikzcd}
	\end{equation*}
	be a Bott tower of height $n$. Let $\gamma_j$ be the tautological line bundle over $B_j = P(\underline{\C} \oplus \xi_j)$. For integers $j,k$ with $1 \leq j \leq k \leq n$, we define 
	\begin{equation*}
		\begin{split}
			x_j^{(k)} &:= (\pi_{j+1} \circ \dots \circ \pi_{k})^*(c_1(\gamma_j)) \in H^2(B_k), \\
			\alpha_j^{(k)} & := (\pi_{j} \circ \dots \circ \pi_k)^*(c_1(\xi_j)) \in H^2(B_k). 
		\end{split}
	\end{equation*}
	By applying Lemma \ref{lemm:Leray-Hirsch} eventually, we have the following.
	\begin{theo}\label{theo:cohomology}
		Let $B_k$ be the Bott manifold as above. The followings hold:
		\begin{enumerate}
			\item $\alpha_j^{(k)}$ is a linear combination of $x_\ell^{(k)}$, $\ell =1, \dots, j-1$. 
			\item $H^*(B_k) = \Z[x_1^{(k)},\dots, x_k^{(k)}]/((x_j^{(k)})^2 - \alpha_j^{(k)}x_j^{(k)} \mid j=1,\dots, k)$.  
		\end{enumerate}
	\end{theo}
	Put $x_j := x_j^{(n)}$.
	It follows from Theorem \ref{theo:cohomology} that $H^*(B_n)$ is generated by degree $2$ elements $x_1, \dots, x_n$ and they form a basis of $H^2(B_n)$. We call them the standard generators of $H^*(B_n)$ with respect to $\xi_1,\dots, \xi_n$.
	
	Suppose that
		\begin{equation*}
			\alpha_j^{(k)} = \sum_{\ell=1}^{j-1} a_j^\ell x_{\ell}^{(k)}, \quad a_j^\ell \in \Z.
		\end{equation*}
	Let $\gamma_j^{(k)}$ be the line bundle over $B_k$ given by
	\begin{equation*}
		\gamma_j^{(k)} : =(\pi_{j+1} \circ \dots \circ \pi_{k})^* \gamma_j. 
	\end{equation*}
	Since $c_1(\gamma_j^{(k)}) = x_j^{(k)}$ and isomorphism classes of line bundles are distinguished by their first Chern classes, we have that the line bundle $\xi_j$ over $B_{j-1}$ is isomorphic to $(\gamma_1^{(j-1)})^{\otimes a_j^1} \otimes \dots \otimes (\gamma_{j-1}^{(j-1)})^{\otimes a_j^{j-1}}$. 
	For simplicity, we assume that $\xi_j = (\gamma_1^{(j-1)})^{\otimes a_j^1} \otimes \dots \otimes (\gamma_{j-1}^{(j-1)})^{\otimes a_j^{j-1}}$. Let $\rho_j \co (S^1)^k \to S^1$ be the homomorphism given by 
	$\rho_j(t) = \prod_{\ell=1}^{j-1}t_\ell^{a_j^\ell}$
	for $t = (t_1,\dots, t_k) \in (S^1)^k$.
	We use the same symbol even if $k$ is different. Let $M_k$ be the quotient manifold $(S^3)^k/(S^1)^k$ by the action of $(S^1)^k$ on $(S^3)^k$ given by
	\begin{equation*}
		\begin{split}
			&(t_1,\dots, t_k) \cdot ((z_1,w_1), \dots, (z_k, w_k))\\
			& = ((t_1z_1, t_1\rho_1(t)^{-1}w_1), \dots, (t_kz_1, t_k\rho_k(t)^{-1}w_k)). 
		\end{split}
	\end{equation*}
	Let $p_k \co M_k \to M_{k-1}$ be the map induced by the projection $(S^3)^k \to (S^3)^{k-1}$. Then $p_k \co M_k \to M_{k-1}$ is a $\C P^1$-bundle. Therefore we have the iterated $\C P^1$-bundle 
	\begin{equation*}
		\begin{tikzcd}
			M_\bullet \co  M_n \arrow[r, "p_n"] & M_{n-1} \ar[r, "p_{n-1}"] & \cdots \ar[r, "p_2"] & M_1 \ar[r, "p_1"] & M_0 = \{\text{a point}\}. 
		\end{tikzcd}
	\end{equation*}
	Let $L_k$ be the quotient $((S^3)^k \times \C)/(S^1)^k$ by the action of $(S^1)^k$ on $(S^3)^k \times \C$ given by
	\begin{equation*}
		\begin{split}
			&(t_1,\dots, t_k) \cdot ((z_1,w_1), \dots, (z_k, w_k), v)\\
			& = ((t_1z_1, t_1\rho_1(t)^{-1}w_1), \dots, (t_kz_1, t_k\rho_k(t)^{-1}w_k), t_k^{-1}v). 
		\end{split}
	\end{equation*}
	For a homomorphism $\rho \co (S^1)^k \to S^1$, let $L_\rho$ be the $((S^3)^k \times \C)/(S^1)^k$ by the action of $(S^1)^k$ on $(S^3)^k \times \C$ given by
	\begin{equation*}
		\begin{split}
			&(t_1,\dots, t_k) \cdot ((z_1,w_1), \dots, (z_k, w_k), v)\\
			& = ((t_1z_1, t_1\rho_1(t)^{-1}w_1), \dots, (t_kz_1, t_k\rho_k(t)^{-1}w_k), \rho(t)^{-1}v). 
		\end{split}
	\end{equation*}
	$L_k$ and $L_\rho$ are line bundles over $M_k$. 
	\begin{prop}[{See also \cite[Section 3]{CR2005}}]\label{prop:quotientconstruction}
		For all $k$ there exists a diffeomorphism $f_k \co M_k \to B_k$ such that $\pi_k\circ f_k = f_{k-1}\circ p_k$. Moreover, there exists a line bundle isomorphism $\widetilde f_k \co L_k \to \gamma_k$ which induces $f_k$. 
	\end{prop}
	\begin{proof}
		Induction on $k$. Consider the case when $k=1$. Since $\xi_1$ is a line bundle over a point, we have that $\xi_1$ is the product line bundle and $a_1^{(0)} =0$. Therefore $\rho_1(t) =1$ and hence $B_1$ and $M_1$ are nothing but $\C P^1$. 
		
		The tautological line bundle $\gamma$ of $\C P^1$ is 
		\begin{equation*}
			\gamma = \{ ([z,w], (u,v)) \in \C P^1 \times \C^2 \mid (u,v) \in [z,w]\}.
		\end{equation*}
		The smooth map $S^3 \times \C \to S^3 \times \C^2$ given by
			$((z,w), \lambda) \mapsto ((z,w), (\lambda z, \lambda w))$
		induces the isomorphism $\widetilde{f}_1$ of line bundles between $L_1 \to M_1 = B_1$ and $\gamma \to B_1$. 
		
		Suppose that the proposition holds for $k-1$. For $\ell =1,\dots, k-1$, let $L_\ell'$ be the pull-back line bundle $(p_{k-1} \circ \dots \circ p_{\ell+1})^*L_\ell$. By the induction hypothesis, $L_\ell'$ is isomorphic to $\gamma_\ell^{(k-1)}$. Since $\xi_{k}$ is isomorphic to $(\gamma_1^{(k-1)})^{\otimes a_k^1} \otimes \dots \otimes (\gamma_{k-1}^{(k-1)})^{\otimes a_k^{k-1}}$, we have that
		\begin{equation*}
			f_{k-1}^*(\xi_k) \cong (L_1')^{\otimes a_k^1} \otimes \dots \otimes (L_{k-1}')^{\otimes a_k^{k-1}}. 
		\end{equation*}
		On the other hand, $(L_1')^{\otimes a_k^1} \otimes \dots \otimes (L_{k-1}')^{\otimes a_k^{k-1}}$ is isomorphic to $L_{\rho_k}$ as line bundles over $M_{k-1}$. Therefore there exists a line bundle isomorphism $g_k \co L_{\rho_k} \to \xi_k$ that induces $f_{k-1}$. The vector bundle isomorphism between $\underline{\C}\oplus L_{\rho_k}$ and $\underline{\C}\oplus \xi_k$ induces the diffeomorphism $f_k \co M_k \to B_k$ such that $\pi_k \circ f_k = f_{k-1}\circ p_k$. 
		The smooth map $(S^3)^k \times \C \to (S^3)^k \times \C^2$ given by
		\begin{equation*}
			((z_1,w_1), \dots , (z_k,w_k), \lambda) \mapsto ((z_1,w_1), \dots , (z_k,w_k), (\lambda z_k, \lambda w_k))
		\end{equation*}
		induces the isomorphism between $L_k$ and the tautological line bundle of $P(\underline{\C}\oplus L_{\rho_k}) = M_k$. Remark that $f_k^*(\gamma_k)$ is isomorphic to the tautological line bundle of $M_k$. By composing, we have that there exists a line bundle isomorphism $\widetilde f_k \co L_k \to \gamma_k$ which induces $f_k$. 
		
		The proposition is proved. 
	\end{proof}
	\begin{theo}[{\cite[Theorem 3.1]{Ishida2012}}]\label{theo:decomposable}
		Rank $2$ decomposable vector bundles over a Bott manifold are distinguished by their total Chern classes. 
	\end{theo}
	Let 
	\begin{equation*}
		\begin{tikzcd}
			B_\bullet' \co  B_n' \arrow[r, "\pi_n'"] & B_{n-1}' \ar[r, "\pi_{n-1}'"] & \cdots \ar[r, "\pi_2'"] & B_1' \ar[r, "\pi_1'"] & B_0' = \{\text{a point}\}
		\end{tikzcd}
	\end{equation*}
	be another Bott tower of height $n$ such that $\pi_j' \co B_j' \to B_{j-1}'$ is a projectivization $P(\underline{\C} \oplus \xi') \to B_{j-1}'$, where $\xi_j'$ is a complex line bundle over $B_{j-1}'$. Let $x_1',\dots, x_n'$ be the standard generators of $H^*(B_n')$ with respect to $\xi_1',\dots, \xi_n'$. 
	\begin{theo}[{\cite[Theorem 1.1]{Ishida2012}}]\label{theo:uppertriangular}
		Let $\varphi \co H^*(B_n) \to H^*(B_n')$ be an isomorphism as graded algebras. Assume that the representation matrix of $\varphi$ with respect to $x_1,\dots, x_n$ and $x_1',\dots, x_n'$ is an upper triangular matrix. Then, there exists a diffeomorphism $f_k \co B_k' \to B_k$ such that $\pi_k \circ f_k = f_{k-1}\circ \pi_k'$ for all $k$. 
	\end{theo}

\section{The Hirzebruch surfaces}\label{sec:Hirzebruch}
	Let $a \in \Z$. The Hirzebruch surface $\Sigma_a$ is the total space of the projectivization of the rank $2$ vector bundle over $\C P^1$: 
	\begin{equation*}
		\Sigma _a = P (\C \oplus \gamma^{\otimes a}) \to \C P^1.
	\end{equation*}
	In particular, $\Sigma_a$ is a Bott manifold of dimension $4$. 
	By Theorem \ref{theo:cohomology}, the cohomology algebra $H^*(\Sigma_a)$ is of the form
	\begin{equation*}
		H^*(\Sigma_a) = \Z [x_1, x_2]/(x_1^2, x_2(x_2-ax_1)). 
	\end{equation*}
	The strong cohomological rigidity for Hirzebruch surfaces is already shown in \cite{CM2012}. 
	In this section, we review it briefly and study the certain diffeomorphism of $\Sigma_a$. 
	
	First we shall review the classification of the diffeomorphism types of Hirzebruch surfaces given in \cite{Hirzebruch1951} briefly. Let $a, b \in \Z$. By direct computation the isomorphism type of the cohomology algebra $H^*(\Sigma_a)$ is determined by the parity of $a$. In case when $a$ is even, say $a=2k$, then we have that 
	\begin{equation*}
		\begin{split}
			\Sigma_{2k} &= P(\underline{\C} \oplus \gamma^{\otimes 2k}) \\
			&\cong P(\gamma^{\otimes (-k)} \otimes \gamma^{\otimes k}) \quad (\text{by Lemma \ref{lemm:tensor}})\\
			&\cong P(\underline{\C} \oplus \underline{\C}) \quad (\text{by Theorem \ref{theo:decomposable}})\\
			&= \Sigma_0 = \C P ^1 \times \C P^1.
		\end{split}
	\end{equation*}
	In case when $a$ is odd, say $a=2k+1$, then we have that
	\begin{equation*}
		\begin{split}
			\Sigma_{2k+1} &= P(\underline{\C} \oplus \gamma^{\otimes 2k+1}) \\
			&\cong P(\gamma^{\otimes (-k)} \otimes \gamma^{\otimes k+1}) \quad (\text{by Lemma \ref{lemm:tensor}})\\
			&\cong P(\underline{\C} \oplus \gamma) \quad (\text{by Theorem \ref{theo:decomposable}})\\
			&= \Sigma_1.
		\end{split}
	\end{equation*}
	Thus the diffeomorphism types of Hirzebruch surfaces are determined by the parity of $a$, and hence they are determined by the isomorphism types of the cohomologies. 
	
	We shall see the group of automorphisms of $H^*(\Sigma_a)$. Since $x_1$ and $x_2$ form a basis of $H^2(\Sigma_a)$ and $H^*(\Sigma_a)$ is generated by $x_1$ and $x_2$, the automorphism of $H^*(\Sigma_a)$ is determined by the images of $x_1$ and $x_2$. By direct computations one can see that there are $8$ automorphisms of $H^*(\Sigma_a)$.  Among them, there are $4$ diffeomorphisms whose representation matrices with respect to $x_1,x_2$ are the following upper triangular matrices
	\begin{equation*}
		\begin{pmatrix}
			1 & 0 \\
			0 & 1
		\end{pmatrix},
		\begin{pmatrix}
			-1 & 0 \\
			0 & -1
		\end{pmatrix},
		\begin{pmatrix}
			1 & a \\
			0 & -1
		\end{pmatrix},
		\begin{pmatrix}
			-1 & -a \\
			0 & 1
		\end{pmatrix}.
	\end{equation*}
	Remaining $4$ automorphisms have different forms by the parity of $a$:
	\begin{itemize}
		\item In case when $a$ is even, representation matrices are 
		\begin{equation*}
			\pm \begin{pmatrix}
				\frac{a}{2} & \frac{a^2}{4}-1\\
				-1 & -\frac{a}{2}
			\end{pmatrix}, 
			\pm \begin{pmatrix}
				\frac{a}{2} & \frac{a^2}{4}+1\\
				-1 & -\frac{a}{2}
			\end{pmatrix}. 
		\end{equation*}
		\item In case when $a$ is odd, representation matrices are 
		\begin{equation*}
			\pm \begin{pmatrix}
				a & \frac{a^2-1}{2}\\
				-2 & -a
			\end{pmatrix}, 
			\pm \begin{pmatrix}
				a & \frac{a^2+1}{2}\\
				-2 & -a
			\end{pmatrix}. 
		\end{equation*}
	\end{itemize}
	To show the strong cohomological rigidity for Hirzebruch surfaces, we need the following: for any automorphism $\varphi$ of $H^*(\Sigma_0)$ (respectively, $H^*(\Sigma_1)$), there exists a diffeomorphism $f$ of $\Sigma_0$ (respectively, $\Sigma_1$) such that $f^* = \varphi$. 
	
	For $\ell \in \C P^1$, we denote by $\ell^\perp \in \C P^1$ the orthogonal complement of $\ell$ in $\C ^2$. Let $\varphi \co H^*(\Sigma_0) \to H^*(\Sigma_0)$ be an isomorphism. The representation matrix of $\varphi$ with respect to $x_1$, $x_2$ is a signed permutation matrix and vice versa. A signed permutation matrix of size $2$ is the identity matrix or a multiplication of several $\begin{pmatrix} 0 & 1\\ 1 & 0 \end{pmatrix}$ and $\begin{pmatrix} -1 & 0 \\ 0 &1 \end{pmatrix}$. Therefore $\varphi$ is induced by the identity map or a composition of diffeomorphisms $(\ell_1, \ell_2) \mapsto (\ell_2, \ell_1)$ and $(\ell_1, \ell_2) \mapsto (\ell_1^\perp, \ell_2)$ for $(\ell_1,\ell_2) \in \Sigma_0 = \C P^1 \times \C P^1$. 
	
	Let $a = \pm 1$. We will construct a diffeomorphism of $\Sigma_a$ explicitly for all automorphism of $H^*(\Sigma_a)$. For a moment, we use column vector notation to represent elements in $\C ^2$.
	By Proposition \ref{prop:quotientconstruction}, $\Sigma_a$ is diffeomorphic to the quotient $(S^3) ^2 /(S^1)^2$ by the action of $(S^1)^2$ on $(S^3)^2$ given by
	\begin{equation*}
		\begin{pmatrix}
			t_1\\ t_2 
		\end{pmatrix} \cdot 
		(\begin{pmatrix}
			z_1 \\ w_1
		\end{pmatrix}, 
		\begin{pmatrix}
			z_2 \\ w_2
		\end{pmatrix})
		= (\begin{pmatrix}
			t_1z_1 \\ t_1w_1
		\end{pmatrix}, 
		\begin{pmatrix}
			t_2z_2 \\ t_1^{-a}t_2w_2
		\end{pmatrix}). 
	\end{equation*}
	We identify $\Sigma_a$ with the quotient $(S^3)^2/(S^1)^2$ via the diffeomorphism. We denote by $\begin{bmatrix}
		\begin{pmatrix}
			z_1 \\ w_1
		\end{pmatrix}, 
		\begin{pmatrix}
			z_2 \\ w_2
		\end{pmatrix}
		\end{bmatrix} \in \Sigma_a$ the equivalence class of $(\begin{pmatrix}
			z_1 \\ w_1
		\end{pmatrix}, 
		\begin{pmatrix}
			z_2 \\ w_2
		\end{pmatrix}) \in (S^3)^2$. 
	For a complex number $z$ and $k \in \Z$, we define
	\begin{equation*}
		z^{[k]} :=\begin{cases}
			z^k & \text{if $k>0$}, \\
			1 & \text{if $k=0$}, \\
			\overline{z}^{-k} & \text{if $k<0$}. 
		\end{cases}
	\end{equation*}
	Let $f \co \Sigma_a \to \Sigma_a$ be the map given by
	\begin{equation*}
		\begin{bmatrix}
		\begin{pmatrix}
			z_1 \\ w_1
		\end{pmatrix}, 
		\begin{pmatrix}
			z_2 \\ w_2
		\end{pmatrix}
		\end{bmatrix} \mapsto
		\begin{bmatrix}
		(|z_2|^4+|w_2|^4)^{-1/2}
		\begin{pmatrix}
			z_1z_2^{[-2a]} - \overline{w_1}w_2^{[-2a]} \\ w_1z_2^{[-2a]}+\overline{z_1}w_2^{[-2a]}
		\end{pmatrix}, 
		\begin{pmatrix}
			\overline{z_2} \\ w_2
		\end{pmatrix}
		\end{bmatrix}.
	\end{equation*}
	By direct computation we have that $f$ is well-defined. Moreover, $f$ admits a smooth inverse, because 
	\begin{equation*}(|z_2|^4+|w_2|^4)^{-1/2}
		\begin{pmatrix}
			z_1z_2^{[-2a]} - \overline{w_1}w_2^{[-2a]} \\ w_1z_2^{[-2a]}+\overline{z_1}w_2^{[-2a]}
		\end{pmatrix}
	\end{equation*}
	is the first column vector of the special unitary matrix
	\begin{equation*}
		\begin{pmatrix}
			z_1 & -\overline{w_1}\\
			w_1 & \overline{z_1}
		\end{pmatrix}
		(|z_2|^4+|w_2|^4)^{-1/2}
		\begin{pmatrix}
			z_2^{[-2a]} & -w_2^{[2a]}\\
			w_2^{[-2a]} & z_2^{[2a]}
		\end{pmatrix}.
	\end{equation*}
	We claim that $f^*(x_1) = -2ax_2+x_1$. To see this, we consider the pull-back $L$ of the tautological line bundle of $\C P^1$ by $\pi_2 \co \Sigma_a \to \C P^1$, that is, 
	\begin{equation*}
		L = \{ (\begin{bmatrix}
		\begin{pmatrix}
			z_1 \\ w_1
		\end{pmatrix}, 
		\begin{pmatrix}
			z_2 \\ w_2
		\end{pmatrix}
		\end{bmatrix}, \begin{pmatrix} u\\ v\end{pmatrix}) \in \Sigma_a \times \C^2 \mid {}^\exists \lambda \in \C \text{ such that } \begin{pmatrix} u \\ v\end{pmatrix} = \begin{pmatrix} \lambda z_1 \\ \lambda w_1 \end{pmatrix}\}. 
	\end{equation*}
	Then the pull-back $f^*(L)$ is 
	\begin{equation*}
		\{ (\begin{bmatrix}
		\begin{pmatrix}
			z_1 \\ w_1
		\end{pmatrix}, 
		\begin{pmatrix}
			z_2 \\ w_2
		\end{pmatrix}
		\end{bmatrix}, \begin{pmatrix} u\\ v\end{pmatrix}) \in \Sigma_a \times \C^2 \mid {}^\exists \lambda \in \C \text{ such that } \begin{pmatrix} u \\ v\end{pmatrix} = \begin{pmatrix} \lambda (z_1z_2^{[-2a]} - \overline{w_1}w_2^{[-2a]}) \\ \lambda (w_1z_2^{[-2a]}+\overline{z_1}w_2^{[-2a]}) \end{pmatrix}\}. 
	\end{equation*}
	On the other hand, $f^*(L)$ is isomorphic to $L_1' \otimes L_2^{\otimes (-2a)}$, where $L_1'$ and $L_2$ are defined as in the proof of Proposition \ref{prop:quotientconstruction}. In fact, the map $L_1' \otimes L_2^{\otimes (-2a)} \to f^*(L)$ given by
	\begin{equation*}
		\begin{bmatrix}
		\begin{pmatrix}
			z_1 \\ w_1
		\end{pmatrix}, 
		\begin{pmatrix}
			z_2 \\ w_2
		\end{pmatrix}, \lambda
		\end{bmatrix} \mapsto (\begin{bmatrix}
		\begin{pmatrix}
			z_1 \\ w_1
		\end{pmatrix}, 
		\begin{pmatrix}
			z_2 \\ w_2
		\end{pmatrix}
		\end{bmatrix}, \begin{pmatrix} \lambda (z_1z_2^{[-2a]} - \overline{w_1}w_2^{[-2a]}) \\ \lambda (w_1z_2^{[-2a]}+\overline{z_1}w_2^{[-2a]}) \end{pmatrix})
	\end{equation*} is a line bundle isomorphism. 
	By definition, $c_1(L) = x_1$. The argument above shows that $c_1(f^*(L)) = -2ax_2+x_1$. Therefore $f^*(x_1) = -2ax_2+x_1$. By the same argument we have that $f^*(x_2) = -x_2$. 
	Let $g_1 \co \Sigma_a \to \Sigma_a$ be the diffeomorphism given by
	\begin{equation*}
		\begin{bmatrix}
		\begin{pmatrix}
			z_1 \\ w_1
		\end{pmatrix}, 
		\begin{pmatrix}
			z_2 \\ w_2
		\end{pmatrix}
		\end{bmatrix} \mapsto
		\begin{bmatrix}
		\begin{pmatrix}
			z_1 \\ w_1
		\end{pmatrix}, 
		\begin{pmatrix}
			-\overline{w_2} \\ \overline{z_2}
		\end{pmatrix}
		\end{bmatrix}.
	\end{equation*}
	Then the representation matrix of $g_1^*$ is $\begin{pmatrix} 1 & a \\ 0 & -1 \end{pmatrix}$. 
	Let $g_2 \co \Sigma_a \to \Sigma_a$ be the diffeomorphism given by 
	\begin{equation*}
		\begin{bmatrix}
		\begin{pmatrix}
			z_1 \\ w_1
		\end{pmatrix}, 
		\begin{pmatrix}
			z_2 \\ w_2
		\end{pmatrix}
		\end{bmatrix} \mapsto
		\begin{bmatrix}
		\begin{pmatrix}
			-\overline{w_1} \\ \overline{z_1}
		\end{pmatrix}, 
		\begin{pmatrix}
			w_2 \\ z_2
		\end{pmatrix}
		\end{bmatrix}.
	\end{equation*}
	Then the representation matrix of $g_2^*$ is $\begin{pmatrix} -1 & -a \\ 0 & 1\end{pmatrix}$. 
	These computations show that the diffeomorphism $\varphi$ is induced by one of the identity map or a composition of several $f$, $g_1$ and $g_2$.

	Finally, we remark that $f$, $g_1$ and $g_2$ are equivariant with respect to the following $S^1$-action on $\Sigma_a$. We define the $S^1$-action on $\Sigma_a$ via 
	\begin{equation*}
		t \cdot 
		\begin{bmatrix}
		\begin{pmatrix}
			z_1 \\ w_1
		\end{pmatrix}, 
		\begin{pmatrix}
			z_2 \\ w_2
		\end{pmatrix}
		\end{bmatrix} \mapsto
		\begin{bmatrix}
		\begin{pmatrix}
			z_1 \\ t^2w_1
		\end{pmatrix}, 
		\begin{pmatrix}
			z_2 \\ t^{-a}w_2
		\end{pmatrix}
		\end{bmatrix}
	\end{equation*}
	for $t \in S^1$ and $\begin{bmatrix}
		\begin{pmatrix}
			z_1 \\ w_1
		\end{pmatrix}, 
		\begin{pmatrix}
			z_2 \\ w_2
		\end{pmatrix}
		\end{bmatrix} \in \Sigma_a$. 
	We shall state this as a lemma for later use. 
	\begin{lemm}\label{lemm:equivariantdiffeo}
		Let $a = \pm 1$.  For any automorphism $\varphi \co H^*(\Sigma_a) \to H^*(\Sigma_a)$, there exists an $S^1$-equivariant diffeomorphism $f \co \Sigma_a \to \Sigma_a$ such that $f^*  = \varphi$. 
	\end{lemm}
	\begin{rema}
		Let $a= \pm 1$. It is known that $\Sigma_a$ is diffeomorphic to the connected sum $\C P^2 \# \overline{\C P^2}$. Using certain involutions on $\C P^2$ and focusing on the connected sum, one can construct diffeomorphisms that induce all automorphisms of the cohomology of $\Sigma_a$, see \cite[Proof of Lemma 5.4]{CM2012}. For Lemma \ref{lemm:equivariantdiffeo}, we focus on the quotient construction (Proposition \ref{prop:quotientconstruction}) and construct the equivariant diffeomorphisms.
	\end{rema}

	\section{Hirzebruch surface bundles and algebra automorphisms}\label{sec:auto}
	Let 
	\begin{equation*}
		\begin{tikzcd}
			B_{n+2} \arrow[r, "\pi_{n+2}"] & B_{n+1} \arrow[r, "\pi_{n+1}"]  & B_n \arrow[r, "\pi_n"] & \cdots \ar[r, "\pi_2"] & B_1 \ar[r, "\pi_1"] & B_0 = \{\text{a point}\}.
		\end{tikzcd}
	\end{equation*}
	be a Bott tower of height $n+2$ such that $\pi_j \co B_j \to B_{j-1}$ is a projectivization $P(\underline{\C} \oplus \xi_j) \to B_{j-1}$, where $\xi_j$ is a complex line bundle over $B_{j-1}$. 
	We think of $H^*(B_n)$ and $H^*(B_{n+1})$ as subalgebras of $H^*(B_{n+2})$ via the injections $\pi_{n+1}^*$ and $\pi_{n+2}^*$. 
	
	By applying Lemma \ref{lemm:Leray-Hirsch} twice, we have that $H^*(B_{n+2})$ is freely generated by $x_{n+2} = c_1(\gamma_{n+2})$ and $x_{n+1} = c_1(\gamma_{n+1})$ as an $H^*(B_n)$-algebra, where $\gamma_{n+1}$ and $\gamma_{n+2}$ are tautological line bundles of $B_{n+1} = P(\underline{\C}\oplus \xi_{n+1})$ and $B_{n+2} = P(\underline{\C}\oplus \xi_{n+2})$, respectively. Suppose that $c_1(\xi_{n+2}) = ax_{n+1} + y$ for $a \in \Z$ and $y\in H^2(B_n)$. Then the composed bundle $B_{n+2} \to B_n$ is a fiber bundle with fiber $\Sigma_a$. Moreover, there is a natural isomorphism between $H^*(B_{n+2})/H^>(B_n)$ and $H^*(\Sigma_a)$. Let $\overline{x_1}$ and $\overline{x_2}$ be the image of $x_{n+1}$ and $x_{n+2}$ by the projection $H^2(B_{n+2}) \to H^2(\Sigma_a) \cong H^2(B_{n+2})/H^2(B_n)$, respectively. We study the condition of $a$, $c_1(\xi_{n+1})$ and $y$ for an automorphism $\varphi$ of $H^*(\Sigma_a)$ to extend to an algebra automorphism $\widetilde{\varphi}$ of $H^*(B_{n+2})$ as an $H^*(B_n)$-algebra. Remark that $\varphi$ and $\widetilde{\varphi}$ are determined by their restrictions to $H^2(\Sigma_a)$ and $H^2(B_{n+2})$, respectively. In the sequel, we suppose that $u_1, u_2 \in H^2(B_n)$ and $\widetilde{\varphi} \co H^2(B_{n+2}) \to H^2(B_{n+2})$ is a homomorphism which preserves elements in $H^2(B_n)$.

		\begin{itemize}
			\item[(0-i)] Suppose that the representation matrix of $\varphi$ with respect to $\overline{x_1}, \overline{x_2}$ is $\begin{pmatrix} 
					1 & 0\\ 
					0 & 1
				\end{pmatrix}$. Suppose that $\widetilde\varphi (x_{n+1}) = x_{n+1}+ u_1$ and $\widetilde\varphi(x_{n+2}) = x_{n+2}+ u_2$. 
				Then 
				\begin{equation*}
					\begin{split}
						& \widetilde\varphi(x_{n+1})\widetilde\varphi(x_{n+1}-c_1(\xi_{n+1}))\\
						&= 2x_{n+1}u_1 +u_1(u_1- c_1(\xi_{n+1}))
					\end{split}
				\end{equation*}
				and 
				\begin{equation*}
					\begin{split}
						& \widetilde\varphi (x_{n+2})\widetilde\varphi(x_{n+2}-ax_{n+1}-y)\\
						&= x_{n+2}(2u_2-au_1) -ax_{n+1}u_2 + u_2(u_2-au_1-y). 
					\end{split}
				\end{equation*}
				Therefore $\widetilde{\varphi}$ becomes an automorphism of $H^*(B_n)$ if and only if $u_1 = u_2 = 0$ because $H^*(B_{n+2})$ is freely generated by $1, x_{n+1}, x_{n+2}, x_{n+1}x_{n+2}$ as an $H^*(B_n)$-module.
				These computations show that $\varphi$ always uniquely extends to an automorphism of $H^*(B_{n+2})$ as an $H^*(B_n)$-algebra. 
				\item[(0-ii)] Suppose that the representation matrix of $\varphi$ with respect to $\overline{x_1}, \overline{x_2}$ is $\begin{pmatrix}
					1 & a\\
					0 & -1
				\end{pmatrix}$. Suppose that $\widetilde\varphi (x_{n+1}) = x_{n+1}+ u_1$ and $\widetilde\varphi(x_{n+2}) = -x_{n+2}+ax_{n+1}+u_2$. Then
				\begin{equation*}
					\begin{split}
						 & \widetilde\varphi(x_{n+1})\widetilde\varphi(x_{n+1}-c_1(\xi_{n+1}))\\
						&= 2x_{n+1}u_1+u_1(u_1-c_1(\xi_{n+1}))
					\end{split}
				\end{equation*}
				and 
				\begin{equation*}
					\begin{split}
						&\widetilde\varphi (x_{n+2})\widetilde\varphi(x_{n+2}-ax_{n+1}-y)\\
						&= x_{n+2}(au_1-2u_2+2y) + ax_{n+1}(u_2-au_1-y)+u_2(u_2-au_1-y).
					\end{split}
				\end{equation*}
				Therefore $\widetilde{\varphi}$ becomes an automorphism of $H^*(B_{n+2})$ if and only if $u_1=0$ and $u_2=y$. 
				
				These computations show that $\varphi$ always uniquely extends to an automorphism of $H^*(B_{n+2})$ as an $H^*(B_n)$-algebra. 
		\end{itemize}
		For other $6$ automorphisms of $H^*(\Sigma_a)$, we need to separate cases by the value of $a$. 
		\begin{enumerate}
			\item Suppose that $a=0$. 
			\begin{enumerate}
				\item Suppose that the representation matrix of $\varphi$ with respect to $\overline{x_1}, \overline{x_2}$ is $\begin{pmatrix}
					-1 & 0\\
					0 & -1
				\end{pmatrix}$. Suppose that $\widetilde\varphi (x_{n+1}) = -x_{n+1}+ u_1$ and $\widetilde\varphi(x_{n+2}) = -x_{n+2}+ u_2$. Then
				\begin{equation*}
					\begin{split}
						 &\widetilde\varphi(x_{n+1})\widetilde\varphi(x_{n+1}-c_1(\xi_{n+1}))\\
						&= x_{n+1}(2c_1(\xi_{n+1})-2u_1) +u_1(u_1-c_1(\xi_{n+1}))
					\end{split}
				\end{equation*}
				and
				\begin{equation*}
					\begin{split}
						 &\widetilde\varphi (x_{n+2})\widetilde\varphi(x_{n+2}-y)\\
						&= x_{n+2}(2y -2u_2) + u_2(u_2-y).
					\end{split}
				\end{equation*}
				Therefore $\widetilde{\varphi}$ becomes an automorphism of $H^*(B_{n+2})$ if and only if $u_1 = c_1(\xi_{n+1})$ and $u_2 = y$. 
				
				These computations show that, $\varphi$ always uniquely extends to an automorphism of $H^*(B_{n+2})$ as an $H^*(B_n)$-algebra.

				\item  Suppose that the representation matrix of $\varphi$ with respect to $\overline{x_1}, \overline{x_2}$ is $\begin{pmatrix}
					-1 & 0\\
					0 & 1
				\end{pmatrix}$. Then $\varphi$ always uniquely extends to an automorphism of $H^*(B_{n+2})$ as an $H^*(B_n)$-algebra because (0-ii), (1-i) and $\begin{pmatrix} 
					-1 & 0 \\
					0 & 1
				\end{pmatrix} = \begin{pmatrix}
				-1 & 0 \\
				0 & -1 \end{pmatrix}\begin{pmatrix}
				1 & 0\\
				0 & -1
				\end{pmatrix}$.
				\item Suppose that the representation matrix of $\varphi$ with respect to $\overline{x_1}, \overline{x_2}$ is $\begin{pmatrix}
					0 & -1\\
					-1 & 0
				\end{pmatrix}$. Suppose that $\widetilde\varphi (x_{n+1}) = -x_{n+2}+ u_1$ and $\widetilde\varphi(x_{n+2}) = -x_{n+1}+ u_2$. Then
				\begin{equation*}
					\begin{split}
						 & \widetilde\varphi(x_{n+1})\widetilde\varphi(x_{n+1}-c_1(\xi_{n+1}))\\
						&= x_{n+2}(y -2u_1+c_1(\xi_{n+1}))  +u_1(u_1-c_1(\xi_{n+1}))
					\end{split}
				\end{equation*}
				and
				\begin{equation*}
					\begin{split}
						&\widetilde\varphi (x_{n+2})\widetilde\varphi(x_{n+2}-y)\\
						&= x_{n+1}(c_1(\xi_{n+1}) -2u_2+y)+u_2(u_2-y). 
					\end{split}
				\end{equation*}
				Therefore $\widetilde{\varphi}$ becomes an automorphism of $H^*(B_{n+2})$ if and only if $u_1 = u_2 = (c_1(\xi_{n+1})+y)/2$ and $c_1(\xi_{n+1})^2=y^2$. 
				
				These computations show that, $\varphi$ uniquely extends to an automorphism of $H^*(B_{n+2})$ as an $H^*(B_n)$-algebra if and only if $c_1(\xi_{n+1})^2=y^2$ and $c_1(\xi_{n+1})\pm y$ is even. 
				
				\item Suppose that the representation matrix of $\varphi$ with respect to $\overline{x_1}, \overline{x_2}$ is $\begin{pmatrix}
					0 & 1\\
					1 & 0
				\end{pmatrix}$. 
				Then $\varphi$ uniquely extends to an automorphism of $H^*(B_{n+2})$ as an $H^*(B_n)$-algebra if and only if $c_1(\xi_{n+1})^2=y^2$ and $c_1(\xi_{n+1})\pm y$ is even because (1-i), (1-iii) and 
				\begin{equation*}
					\begin{pmatrix}
						0 & 1\\
						1 & 0
					\end{pmatrix}
					= \begin {pmatrix}
						0 & -1 \\
						-1 & 0
					\end{pmatrix}
					\begin{pmatrix}
						-1 & 0\\
						0 & -1
					\end{pmatrix}.
				\end{equation*}

				\item Suppose that the representation matrix of $\varphi$ with respect to $\overline{x_1}, \overline{x_2}$ is $\pm\begin{pmatrix}
					0 & 1\\
					-1 & 0
				\end{pmatrix}$. Then $ \varphi$ uniquely extends to an automorphism of $H^*(B_{n+2})$ as an $H^*(B_n)$-algebra if and only if  $c_1(\xi_{n+1})^2=y^2$  and $c_1(\xi_{n+1})\pm y$ is even because (0-ii), (1-iii), (1-iv) and 
				\begin{equation*}
					\pm\begin{pmatrix}
						0 & 1\\
						-1 & 0
					\end{pmatrix} = \pm\begin{pmatrix}
						0 & -1\\
						-1 & 0 \end{pmatrix}
					\begin{pmatrix}
						1 & 0\\
						0 & -1
					\end{pmatrix}. 
				\end{equation*}
			\end{enumerate}
			
			\item Suppose that $a$ is nonzero even.
				\begin{enumerate}
					\item Suppose that the representation matrix of $\varphi$ with respect to $\overline{x_1}, \overline{x_2}$ is $\begin{pmatrix}
						-1 & 0\\
						0 & -1
					\end{pmatrix}$. Suppose that $\widetilde\varphi (x_{n+1}) = -x_{n+1}+ u_1$ and $\widetilde\varphi(x_{n+2}) = -x_{n+2}+ u_2$. Then
					\begin{equation*}
						\begin{split}
							 & \widetilde\varphi(x_{n+1})\widetilde\varphi(x_{n+1}-c_1(\xi_{n+1}))\\
							&= x_{n+1}(2c_1(\xi_{n+1})-2u_1) +u_1(u_1-c_1(\xi_{n+1})
						\end{split}
					\end{equation*}
					and
					\begin{equation*}
						\begin{split}
							& \widetilde\varphi (x_{n+2})\widetilde\varphi(x_{n+2}-ax_{n+1}-y)\\
							&= x_{n+2}(2y -2u_2+au_1) +ax_{n+1}u_2 + u_2(u_2-au_1-y).
						\end{split}
					\end{equation*}
					
					Therefore $\widetilde{\varphi}$ becomes an automorphism of $H^*(B_{n+2})$ if and only if $u_1 = c_1(\xi_{n+1})$, $u_2 = 0$ and $y=-\frac{a}{2}c_1(\xi_{n+1})$. 
					
					These computations show that, $\varphi$ uniquely extends to an automorphism of $H^*(B_{n+2})$ as an $H^*(B_n)$-algebra if and only if $y = -\frac{a}{2}c_1(\xi_{n+1})$. 
					\item  Suppose that the representation matrix of $\varphi$ with respect to $\overline{x_1}, \overline{x_2}$ is $\begin{pmatrix}
						-1 & -a\\
						0 & 1
					\end{pmatrix}$. Then $\varphi$ uniquely extends to an automorphism of $H^*(B_{n+2})$ as an $H^*(B_n)$-algebra if and only if $y=-\frac{a}{2}c_1(\xi_{n+1})$  because (0-ii), (2-i) and 
					\begin{equation*}
					\begin{pmatrix} 
						-1 & -a \\
						0 & 1
					\end{pmatrix} = \begin{pmatrix}
					-1 & 0 \\
					0 & -1 \end{pmatrix}\begin{pmatrix}
					1 & a\\
					0 & -1
					\end{pmatrix}.
					\end{equation*}

					\item Suppose that the representation matrix of $\varphi$ with respect to $\overline{x_1}, \overline{x_2}$ is $\begin{pmatrix}
					\frac{a}{2} & \frac{a^2}{4}-1\\
					-1 & -\frac{a}{2}
				\end{pmatrix}$. 
				Suppose that $\widetilde\varphi (x_{n+1}) = -x_{n+2}+\frac{a}{2}x_{n+1}+ u_1$ and $\widetilde\varphi(x_{n+2}) = -\frac{a}{2}x_{n+2} + (\frac{a^2}{4}-1)x_{n+1}+ u_2$. Then
				\begin{equation*}
					\begin{split}
						&\widetilde\varphi(x_{n+1})\widetilde\varphi(x_{n+1}-c_1(\xi_{n+1}))\\
						&= x_{n+2}(y -2u_1+c_1(\xi_{n+1})) \\
						& \quad +ax_{n+1}(u_1 +\frac{a-2}{4}c_1(\xi_{n+1})) +u_1(u_1-c_1(\xi_{n+1}))
					\end{split}
				\end{equation*}
				and  
				\begin{equation*}
					\begin{split}
						& \widetilde\varphi (x_{n+2})\widetilde\varphi(x_{n+2}-ax_{n+1}-y)\\
						&= \frac{a}{4}x_{n+2}(2au_1 + (2-a) y)\\				
						&\quad + x_{n+1}((1-\frac{a^4}{16})c_1(\xi_{n+1})-2u_2+ (\frac{a^2}{4}-1)(-au_1-y))\\
						&\quad +u_2(u_2-au_1-y).
					\end{split}
				\end{equation*}
				Therefore $\widetilde{\varphi}$ becomes an automorphism of $H^*(B_{n+2})$ if and only if $u_1 = \frac{2-a}{4}c_1(\xi_{n+1})$,  $u_2 = \frac{4-a^2}{8}c_1(\xi_{n+1})$, $y=-\frac{a}{2}c_1(\xi_{n+1})$ and $(4-a^2)c_1(\xi_{n+1})^2=0$.

				These computations show that, $\varphi$ uniquely extends to an automorphism of $H^*(B_{n+2})$ as an $H^*(B_n)$-algebra if and only if $\frac{2\pm a}{4}c_1(\xi_{n+1})$ is integral, $y = -\frac{a}{2}c_1(\xi_{n+1})$ and $(4-a^2)c_1(\xi_{n+1})^2=0$.

				\item Suppose that the representation matrix of $\varphi$ with respect to $\overline{x_1}, \overline{x_2}$ is $\begin{pmatrix}
					-\frac{a}{2} & -\frac{a^2}{4}+1\\
					1 & \frac{a}{2}
				\end{pmatrix}$. 
				Suppose that $\widetilde{\varphi}(x_{n+1}) = x_{n+2}-\frac{a}{2}x_{n+1} +u_1$ and $\widetilde{\varphi}(x_{n+2}) = \frac{a}{2}x_{n+2} +(-\frac{a^2}{4}+1)x_{n+1}+u_2$.  
				Then
				\begin{equation*}
					\begin{split}
						&\widetilde\varphi(x_{n+1})\widetilde\varphi(x_{n+1}-c_1(\xi_{n+1}))\\
						&= x_{n+2}(y + 2u_1-c_1(\xi_{n+1})) \\
						& \quad +ax_{n+1}(- u_1 + \frac{2+a}{4}c_1(\xi_{n+1})) +u_1(u_1-c_1(\xi_{n+1}))
					\end{split}
				\end{equation*}
				and
				\begin{equation*}
					\begin{split}
						&\widetilde\varphi (x_{n+2})\widetilde\varphi(x_{n+2}-ax_{n+1}-y)\\
						&= -\frac{a}{4}x_{n+2}(2au_1+(2+a)y)\\
						&\quad + x_{n+1}(2u_2-(1-\frac{a^2}{4})(au_1+y) + (1-\frac{a^4}{16})c_1(\xi_{n+1}))\\
						&\quad +u_2(u_2 -au_1-y).
					\end{split}
				\end{equation*}
				Therefore $\widetilde{\varphi}$ becomes an automorphism of $H^*(B_{n+2})$ if and only if $u_1 = \frac{2+a}{4}c_1(\xi_{n+1})$, $u_2 = -\frac{4-a^2}{8}c_1(\xi_{n+1})$, $y=-\frac{a}{2}c_1(\xi_{n+1})$ and $(4-a^2)c_1(\xi_{n+1})^2=0$.
				
				These computations show that, $\varphi$ uniquely extends to an automorphism of $H^*(B_{n+2})$ as an $H^*(B_n)$-algebra if and only if $\frac{2\pm a}{4}c_1(\xi_{n+1})$ is integral, $y = -\frac{a}{2}c_1(\xi_{n+1})$ and $(4-a^2)c_1(\xi_{n+1})^2=0$. 
				
				\item Suppose that the representation matrix of $\varphi$ with respect to $\overline{x_1}, \overline{x_2}$ is $\pm\begin{pmatrix}
					\frac{a}{2} & \frac{a^2}{4}+1\\
					-1 & -\frac{a}{2}
				\end{pmatrix}$. Then $\varphi$ uniquely extends to an automorphism of $H^*(B_{n+2})$ as an $H^*(B_n)$-algebra if and only if $\frac{2\pm a}{4}c_1(\xi_{n+1})$ is integral, $y = -\frac{a}{2}c_1(\xi_{n+1})$ and $(4-a^2)c_1(\xi_{n+1})^2=0$ because (0-ii), (2-iii), (2-iv) and 
				\begin{equation*}
					\pm\begin{pmatrix}
						\frac{a}{2} & \frac{a^2}{4}+1\\
						-1 & -\frac{a}{2}
					\end{pmatrix} = \pm\begin{pmatrix}
						\frac{a}{2} & \frac{a^2}{4}-1\\
						-1 & -\frac{a}{2} \end{pmatrix}
					\begin{pmatrix}
						1 & a\\
						0 & -1
					\end{pmatrix}. 
				\end{equation*}
			\end{enumerate}

			\item Suppose that $a$ is odd. 
			\begin{enumerate}
				\item Suppose that the representation matrix of $\varphi$ with respect to $\overline{x_1}, \overline{x_2}$ is $\begin{pmatrix}
						-1 & 0\\
						0 & -1
					\end{pmatrix}$. 
					Suppose that $\widetilde\varphi (x_{n+1}) = -x_{n+1}+ u_1$ and $\widetilde\varphi(x_{n+2}) = -x_{n+2}+ u_2$. Then 
					\begin{equation*}
						\begin{split}
							&\widetilde\varphi(x_{n+1})\widetilde\varphi(x_{n+1}-c_1(\xi_{n+1}))\\
							&= x_{n+1}(2c_1(\xi_{n+1})-2u_1) +u_1^2 - u_1c_1(\xi_{n+1})
						\end{split}
					\end{equation*}
					and 
					\begin{equation*}
						\begin{split}
							&\widetilde\varphi (x_{n+2})\widetilde\varphi(x_{n+2}-ax_{n+1}-y)\\
							&= x_{n+2}(2y -2u_2+au_1) +ax_{n+1}u_2 + u_2(u_2-au_1-y).
						\end{split}
					\end{equation*}
					Therefore $\widetilde{\varphi}$ becomes an automorphism of $H^*(B_{n+2})$ if and only if $u_1 = c_1(\xi_{n+1})$, $u_2 = 0$ and $y=-\frac{a}{2}c_1(\xi_{n+1})$. 
					
					These computations show that, $\varphi$ uniquely extends to an automorphism of $H^*(B_{n+2})$ as an $H^*(B_n)$-algebra if and only if $c_1(\xi_{n+1})$ is even and $y = -\frac{a}{2}c_1(\xi_{n+1})$. 
					
					\item Suppose that the representation matrix of $\varphi$ with respect to $\overline{x_1}, \overline{x_2}$ is $\begin{pmatrix}
					-1 & -a\\
					0 & 1
				\end{pmatrix}$. Then $\varphi$ uniquely extends to an automorphism of $H^*(B_{n+2})$ as an $H^*(B_n)$-algebra if and only if  $c_1(\xi_{n+1})$ is even and $y = -\frac{a}{2}c_1(\xi_{n+1})$ because (0-ii), (3-i) and 
				\begin{equation*}
				\begin{pmatrix} 
					-1 & -a \\
					0 & 1
				\end{pmatrix} = \begin{pmatrix}
				-1 & 0 \\
				0 & -1 \end{pmatrix}\begin{pmatrix}
				1 & a\\
				0 & -1
				\end{pmatrix}.
				\end{equation*}
				
				\item Suppose that the representation matrix of $\varphi$ with respect to $\overline{x_1}, \overline{x_2}$ is $\begin{pmatrix}
					a & \frac{a^2-1}{2}\\
					-2 & -a
				\end{pmatrix}$. 
				Suppose that $\widetilde\varphi (x_{n+1}) = -2x_{n+2}+{a}x_{n+1}+ u_1$ and $\widetilde\varphi(x_{n+2}) = -{a}x_{n+2} + (\frac{a^2-1}{2})x_{n+1}+ u_2$. Then
				\begin{equation*}
					\begin{split}
						&\widetilde\varphi(x_{n+1})\widetilde\varphi(x_{n+1}-c_1(\xi_{n+1}))\\
						&= 2x_{n+2}(2y -2u_1+c_1(\xi_{n+1})) \\
						& \quad +ax_{n+1}(2u_1+(a-1) c_1(\xi_{n+1})) +u_1(u_1-c_1(\xi_{n+1}))
					\end{split}
				\end{equation*}
				and
				\begin{equation*}
					\begin{split}
						&\widetilde\varphi (x_{n+2})\widetilde\varphi(x_{n+2}-ax_{n+1}-y)\\
						&= ax_{n+2}(au_1+(1-a)y) \\
						&\quad + x_{n+1}(-u_2+(\frac{1-a^2}{2})(au_1+y)+\frac{1-a^4}{4}c_1(\xi_{n+1}))\\
						&\quad +u_2(u_2 -au_1-y).
					\end{split}
				\end{equation*}
				Therefore $\widetilde{\varphi}$ becomes an automorphism of $H^*(B_{n+2})$ if and only if $u_1 = \frac{1-a}{2}c_1(\xi_{n+1})$, $u_2 = \frac{1-a^2}{4}c_1(\xi_{n+1})$, $y=-\frac{a}{2}c_1(\xi_{n+1})$ and $(1-a^2)c_1(\xi_{n+1})^2=0$.
				
				These computations show that, $\varphi$ uniquely extends to an automorphism of $H^*(B_{n+2})$ as an $H^*(B_n)$-algebra if and only if $c_1(\xi_{n+1})$ is even, $y = -\frac{a}{2}c_1(\xi_{n+1})$ and $(1-a^2)c_1(\xi_{n+1})^2=0$. 				
				\item Suppose that the representation matrix of $\varphi$ with respect to $\overline{x_1}, \overline{x_2}$ is $\begin{pmatrix}
					-a & -\frac{a^2-1}{2}\\
					2 & a
				\end{pmatrix}$. 
				Suppose that $\widetilde\varphi (x_{n+1}) = 2x_{n+2}-{a}x_{n+1}+ u_1$ and $\widetilde\varphi(x_{n+2}) = {a}x_{n+2} - (\frac{a^2-1}{2})x_{n+1}+ u_2$. Then
				\begin{equation*}
					\begin{split}
						&\widetilde\varphi(x_{n+1})\widetilde\varphi(x_{n+1}-c_1(\xi_{n+1}))\\
						&= 2x_{n+2}(2y +2u_1-c_1(\xi_{n+1})) \\
						& \quad +ax_{n+1}(-2u_1 + (a+1)c_1(\xi_{n+1})) +u_1(u_1-c_1(\xi_{n+1}))
					\end{split}
				\end{equation*}
				and
				\begin{equation*}
					\begin{split}
						&\widetilde\varphi (x_{n+2})\widetilde\varphi(x_{n+2}-ax_{n+1}-y)\\
						&= -ax_{n+2}(au_1+(a+1)y)\\
						&\quad + x_{n+1}(u_2+\frac{a^2-1}{2}(au_1+y)+\frac{1-a^4}{4}c_1(\xi_{n+1}))\\
						&\quad +u_2(u_2 -au_1-y).
					\end{split}
				\end{equation*}
				Therefore $\widetilde{\varphi}$ becomes an automorphism of $H^*(B_{n+2})$ if and only if $u_1 = \frac{1+a}{2}c_1(\xi_{n+1})$, $u_2 = -\frac{1-a^2}{4}c_1(\xi_{n+1})$, $y= -\frac{a}{2}c_1(\xi_{n+1})$ and $(1-a^2)c_1(\xi_{n+1})^2=0$.
				
				These computations show that, $\varphi$ uniquely extends to an automorphism of $H^*(B_{n+2})$ as an $H^*(B_n)$-algebra if and only if $c_1(\xi_{n+1})$ is even, $y = -\frac{a}{2}c_1(\xi_{n+1})$ and $(1-a^2)c_1(\xi_{n+1})^2=0$. 
				
				\item Suppose that the representation matrix of $\varphi$ with respect to $\overline{x_1}, \overline{x_2}$ is $\pm\begin{pmatrix}
					a & \frac{a^2+1}{2}\\
					-2 & -a
				\end{pmatrix}$. Then $\varphi$ uniquely extends to an automorphism of $H^*(B_{n+2})$ as an $H^*(B_n)$-algebra if and only if $c_1(\xi_{n+1})$ is even, $y = -\frac{a}{2}c_1(\xi_{n+1})$ and $(1-a^2)c_1(\xi_{n+1})^2=0$ because (0-ii), (3-iii), (3-iv) and 
				\begin{equation*}
					\pm\begin{pmatrix}
					a & \frac{a^2+1}{2}\\
					-2 & -a
				\end{pmatrix} =\pm\begin{pmatrix}
					a & \frac{a^2-1}{2}\\
					-2 & -a
				\end{pmatrix}
					\begin{pmatrix}
						1 & a\\
						0 & -1
					\end{pmatrix}. 
				\end{equation*}
			\end{enumerate}
		\end{enumerate}
		\section{Realizing bundle automorphisms} \label{sec:realizing}
		We use the same notations as the previous section. In this section, we show that for any automorphism $\widetilde{\varphi} \co H^*(B_{n+2}) \to H^*(B_{n+2})$ as an $H^*(B_n)$-algebra there exists a bundle automorphism $\widetilde{f} \co B_{n+2} \to B_{n+2}$ over $B_n$ such that $\widetilde{f}^* = \widetilde{\varphi}$. By Theorem \ref{theo:uppertriangular} we know that such $\widetilde{f}$ exists if the representation matrix of the descent homomorphism $\varphi \co H^2(\Sigma_a) \to H^2(\Sigma_a)$ is upper triangular. In the sequel of this section we assume that the representation matrix of $\varphi$ is not upper triangular. 
		For a cohomology class $\alpha$ of degree $2$, we denote by $\gamma_k^\alpha$ a line bundle over $B_k$ such that $c_1(\gamma^\alpha_k) = \alpha$. 
		\begin{enumerate}
			\item Suppose that $a=0$. Then $\xi_{n+2}$ is isomorphic to the pull-back of a line bundle over $B_n$ because $c_1(\xi_{n+2}) = y \in H^2(B_n)$. Thus we may assume that $\xi_{n+2} = \pi_{n+1}^*\gamma_n^{y}$. Then the $\C P^1$-bundle $B_{n+2} \to B_{n+1}$ is the pull-back of $\pi_{n+1}' \co P(\underline{\C} \oplus \gamma_n^y) \to B_n$ by the $\C P^1$-bundle $B_{n+1} \to B_n$. Therefore 
			\begin{equation*}
				B_{n+2} = \{ (\ell_1,\ell_2) \in P(\underline{\C}\oplus \xi_{n+1})\times P(\underline{\C}\oplus \gamma_n^y) \mid \pi_{n+1}(\ell_1) = \pi_{n+1}' (\ell_2)\}.
			\end{equation*}
			By (1-iii), (1-iv) and (1-v) in Section \ref{sec:auto}, we have that $c_1(\xi_{n+1})^2 = y^2$ and $c_1(\xi_{n+1}) \pm y$ is even. Thus we have 
			\begin{equation*}
				\begin{split}
					P(\underline{\C}\oplus \gamma^y_n) 
					&\cong P(\gamma_n^{\frac{1}{2}(c_1(\xi_{n+1})-y)} \oplus \gamma_n^{\frac{1}{2}(c_1(\xi_{n+1})+y)})\\
					&\cong P(\underline{\C}\oplus \xi_{n+1})
				\end{split}
			\end{equation*}
			by Lemma \ref{lemm:tensor} and Theorem \ref{theo:decomposable}.
			Through this bundle isomorphism, $\widetilde{\varphi}$ is induced by one of the maps given by 
			$(\ell_1,\ell_2) \mapsto (\ell_2,\ell_1)$, $(\ell_1,\ell_2) \mapsto (\ell_2^\perp,\ell_1)$, $(\ell_1,\ell_2) \mapsto (\ell_2,\ell_1^\perp)$ or $(\ell_1,\ell_2) \mapsto (\ell_2^\perp,\ell_1^\perp)$. 
			\item Suppose that $a$ is nonzero even. By (2-iii), (2-iv) and (2-v) in Section \ref{sec:auto}, we have that $y = -\frac{a}{2}c_1(\xi_{n+1})$. Since $c_1(\xi_{n+2}) = ax_{n+1} + y$ and $x_{n+1}(x_{n+1}-c_1(\xi_{n+1}))=0$, we have
			\begin{equation*}
				\begin{split}
					P(\underline{\C}\oplus \xi_{n+2}) &\cong P(\underline{\C}\oplus \gamma_{n+1}^{ax_{n+1}+y})\\
					&\cong P(\gamma_{n+1}^{-\frac{a}{2}x_{n+1}} \oplus \gamma_{n+1}^{\frac{a}{2}(x_{n+1}-c_1(\xi_{n+1}))})\\
					&\cong P(\underline{\C} \oplus \gamma_{n+1}^{-\frac{a}{2}c_1(\xi_{n+1})})
				\end{split}
			\end{equation*}
			as bundles over $B_{n+1}$ by Lemma \ref{lemm:tensor} and Theorem \ref{theo:decomposable}. 
			Let $B_{n+2}' = P(\underline{\C} \oplus \gamma_{n+1}^{-\frac{a}{2}c_1(\xi_1)})$. 
			Let $\widetilde{g} \co  B_{n+2} \to B_{n+2}'$ be the bundle isomorphism. Then the composition $(\widetilde{g}^{-1})^* \circ \widetilde{\varphi} \circ \widetilde{g}^* \co H^*(B_{n+2}') \to H^*(B_{n+2}')$ is an automorphism as an $H^*(B_n)$-algebra. 
			This together with (1) yields that $(\widetilde{g}^{-1})^* \circ \widetilde{\varphi} \circ \widetilde{g}^*$ is induced by a bundle automorphism $\widetilde{f'} \co B_{n+2}' \to B_{n+2}'$. Therefore $\widetilde{\varphi}$ is induced by a bundle automorphism $\widetilde{g}^{-1} \circ \widetilde{f'} \circ \widetilde{g}$. 
			
			\item Suppose that $a$ is odd and $a\neq \pm 1$. By (3-iii), (3-iv) and (3-v) in Section \ref{sec:auto}, we have that $c_1(\xi_{n+1})$ is even and $c_1(\xi_{n+1})^2=0$. Thus we have 
			\begin{equation*}
				\begin{split}
					B_{n+1} &= P(\underline{\C} \oplus \xi_{n+1})\\
					 & \cong P(\underline{\C} \oplus \gamma_n^{c_1(\xi_{n+1})})\\
					& \cong P(\gamma_n^{-\frac{1}{2}c_1(\xi_{n+1})} \oplus \gamma_n^{\frac{1}{2}c_1(\xi_{n+1})})\\
					& \cong P(\underline{\C} \oplus \underline{\C})
				\end{split}
			\end{equation*}
			as bundles over $B_n$  by Lemma \ref{lemm:tensor} and Theorem \ref{theo:decomposable}. Let $\pi_{n+1}' \co B_{n+1}' \to B_n$ be the product $\C P^1$-bundle over $B_n$. Let $\widetilde{h} \co B_{n+1}' \to B_{n+1}$ be the bundle isomorphism. Let $\pi_{n+2}' \co B_{n+2}' \to B_{n+1}'$ be the pull-back of $\pi_{n+2} \co B_{n+2} \to B_{n+1}$ by $\widetilde{h}$
			and $\widetilde{g} \co B_{n+2}' \to B_{n+2}$ the bundle isomorphism induced by $\widetilde{h}$. Then the composition $\widetilde{g}^* \circ \widetilde{\varphi} \circ (\widetilde{g}^{-1})^* \co H^*(B_{n+2}') \to H^*(B_{n+2}')$ is an automorphism as an $H^*(B_n)$-algebra. Let $x_{n+1}', x_{n+2}'$ be the first Chern classes of the tautological line bundles of $B_{n+1}'$, $B_{n+2}'$, respectively. 
			Let $\overline{x_{n+1}'}$ and $\overline{x_{n+2}'}$ be the image of $x_{n+1}'$ and $x_{n+2}'$ by the projection $H^2(B_{n+2}') \to H^2(B_{n+2}')/H^2(B_n)$, respectively.
			Since the representation matrix of $\widetilde{g}^* \co H^2(B_{n+2}) \to H^2(B_{n+2}')$ with respect to $x_1,\dots, x_n, x_{n+1}, x_{n+2}$ and $x_1,\dots, x_n, x_{n+1}', x_{n+2}'$ is upper triangular, we have that the descent homomorphism of $\widetilde{g}^* \circ \widetilde{\varphi} \circ (\widetilde{g}^{-1})^*$ with respect to $\overline{x_{n+1}'}, \overline{x_{n+2}'}$ is not upper triangular. 
			By (3-iii), (3-iv), (3-v) and $B_{n+1}' = P(\underline{\C}\oplus \underline{\C})$, we have that there exists $a' \in \Z$ such that $\widetilde{h}^*(c_1(\xi_{n+2})) = a'x_{n+1}'$ . Therefore $\pi_{n+1}' \circ \pi_{n+2}' \co B_{n+2}' \to B_n$ is a trivial bundle with fiber $\Sigma_{a'}$ over $B_n$. Since the Hirzebruch surface $\Sigma_{a'}$ is strongly cohomological rigid as we saw in Section \ref{sec:Hirzebruch}, we have that $\widetilde{g}^* \circ \widetilde{\varphi} \circ (\widetilde{g}^{-1})^*$ is induced by a bundle automorphism $\widetilde{f'}$ of $B_{n+2}' \to B_n$.  Thus $\widetilde{\varphi}$ is induced by the bundle automorphism $\widetilde{g} \circ \widetilde{f'} \circ \widetilde{g}^{-1}$. 
			\item Suppose that $a = \pm 1$. By Lemma \ref{lemm:equivariantdiffeo}, there exists an $S^1$-equivariant diffeomorphism $f$ of $\Sigma_a$ such that $f^* = \varphi$. 
			By (3-iii), (3-iv) and (3-v) in Section \ref{sec:auto}, we have that $c_1(\xi_{n+1})$ is even and $y= -\frac{a}{2}c_1(\xi_{n+1})$. 
			Let $p \co \eta \to B_n$ be a complex line bundle over $B_n$ with a Hermitian metric such that $c_1(\eta) = \frac{1}{2}c_1(\xi_{n+1})$. Let $\{h_\alpha \co p^{-1}(U_\alpha) \to U_\alpha \times \C\}$ be a local trivialization such that each $h_\alpha$ preserves the length of vectors. 
			Let $\{\phi_{\alpha\beta} \co U_{\alpha\beta} \to S^1\}$ be the transition functions of $\eta$ with respect to the open covering $\{U_\alpha\}$. Namely, $h_\alpha\circ h_\beta^{-1} (x,v) = (x, \phi_{\alpha\beta}(x)v)$ for $(x,v) \in U_{\alpha\beta} \times \C$, where $U_{\alpha\beta} = U_{\alpha}\cap U_{\beta}$. 
			Since $c_1(\eta) = \frac{1}{2}c_1(\xi_{n+1})$ we may assume that $\xi_{n+1} = \eta^{\otimes 2}$ and $\xi_{n+2} = \pi_{n+1}^*\eta^{\otimes (-a)} \otimes \gamma_{n+2}^{\otimes a}$. Let $p_2 \co S(\underline{\C} \oplus \xi_{n+2}) \to B_{n+1}$ be the unit sphere bundle. Then $p_2$ is decomposed into the principal $S^1$-bundle $q_2 \co S(\underline{\C} \oplus \xi_{n+2}) \to B_{n+2}$ and $\pi_{n+2} \co B_{n+2} \to B_{n+1}$. Let $p_1 \co S(\underline{\C}\oplus \xi_{n+1}) \to B_n$ be the unit sphere bundle. Then $p_1$ is decomposed into the principal $S^1$-bundle $q_1 \co S(\underline{\C} \oplus \xi_{n+1}) \to B_{n+1}$ and $\pi_{n+1} \co B_{n+1} \to B_{n}$. 
			\begin{equation*}
				\begin{tikzcd}[ampersand replacement = \&]
					S(\underline{\C} \oplus q_1^*\xi_{n+2}) \arrow[d] \arrow[dr] \& \& \\
					S(\underline{\C} \oplus \xi_{n+2}) \arrow[d , "q_2"] \arrow[dr, "p_2"]\& S(\underline{\C} \oplus \xi_{n+1}) \arrow[d, "q_1"] \arrow[dr, "p_2"] .\& \\
					  B_{n+2} \arrow[r, "\pi_{n+2}"] \& B_{n+1} \arrow[r, "\pi_{n+1}"] \& B_n
				\end{tikzcd}
			\end{equation*}
			The pull-back of the $S^3$-bundle $p_2 \co S(\underline{\C} \oplus \xi_{n+2}) \to B_{n+1}$ by $q_1$ is the   unit sphere bundle $S(\underline{\C} \oplus q_1^*\xi_{n+1})$. The composition $S(\underline{\C}\oplus q_1^*\xi_{n+2}) \to B_n$ is the fiber product of the unit sphere bundles $S(\underline{\C}\oplus \eta^{\otimes 2})$ and $S(\underline{\C}\oplus \eta^{\otimes(-a)})$ over $B_n$. Therefore $S(\underline{\C}\oplus q_1^*\xi_{n+2}) \to B_n$ is an $S^3 \times S^3$-bundle over $B_n$. Since $\xi_{n+1} =\eta^{\otimes 2}$ and $q_1^*\xi_{n+2} \cong p_2^*\eta^{\otimes(-a)}$, we have that the $S^3 \times S^3$-bundle $S(\underline{\C}\oplus q_1^*\xi_{n+2}) \to B_n$ has transition functions $\{(1, \phi_{\alpha\beta}^2, 1, \phi_{\alpha\beta}^{-a})\}$. On the other hand, each fiber of $B_{n+2} \to B_n$ is nothing but the quotient of $S^3 \times S^3$ by the $(S^1)^2$-action as we saw in Proposition \ref{prop:quotientconstruction} and Section \ref{sec:Hirzebruch}. Therefore the transition functions of $B_{n+2} \to B_n$ are 
			\begin{equation*}
				\begin{bmatrix}
		\begin{pmatrix}
			z_1 \\ \phi_{\alpha\beta}^2(x)w_1
		\end{pmatrix}, 
		\begin{pmatrix}
			z_2 \\ \phi_{\alpha\beta}^{-a}(x)w_2
		\end{pmatrix}
		\end{bmatrix}
			\end{equation*}
			for $x \in U_{\alpha\beta}$, $\begin{bmatrix}
		\begin{pmatrix}
			z_1 \\ w_1
		\end{pmatrix}, 
		\begin{pmatrix}
			z_2 \\ w_2
		\end{pmatrix}
		\end{bmatrix} \in \Sigma_a$. 
			Thus we have that the $S^1$-equivariant diffeomorphism $f$ of $\Sigma_a$ commutes with the transition functions. 
			Therefore the diffeomorphism $f$ of $\Sigma_a$ extends to the bundle automorphism $\widetilde{f}$ of the $\Sigma_a$-bundle $B_{n+2} \to B_n$.  
		\end{enumerate}
		
		\section{Algebra isomorphisms of Hirzebruch surface bundles}\label{sec:rigidity}
		We use the same notation as Sections \ref{sec:auto} and \ref{sec:realizing}. Let $\xi_{n+1}'$ be a complex line bundle over $B_n$ and $\pi_{n+1}' \co B_{n+1}' \to B_n$ be the projectivization $P(\underline{\C} \oplus \xi_{n+1}') \to B_n$. Let $\xi_{n+2}'$ be a complex line bundle over $B_{n+1}'$ and $\pi_{n+2}' \co B_{n+2}' \to B_{n+1}'$ be the projectivization $P(\underline{\C} \oplus \xi_{n+2}') \to B_{n+1}'$. As well as $H^*(B_{n+2})$, we think of $H^*(B_n)$ and $H^*(B'_{n+1})$ as subalgebras of $H^*(B'_{n+2})$ via the injections $(\pi_{n+1}')^*$ and $(\pi_{n+2}')^*$. Then $H^*(B_{n+2}')$ is freely generated by $x_{n+2}' = c_1(\gamma_{n+2}')$ and $x_{n+1}' = c_1(\gamma_{n+1}')$ as an $H^*(B_n)$-algebra, where $\gamma_{n+2}'$ and $\gamma_{n+1}'$ are tautological line bundles of $B_{n+2}' \to B_{n+1}'$ and $B_{n+1}' \to B_{n}$, respectively. 
		In this section, we show that if $H^*(B_{n+2})$ and $H^*(B_{n+2}')$ are isomorphic as $H^*(B_n)$-algebras then $B_{n+2}$ and $B_{n+2}'$ are isomorphic as bundles over $B_n$. 
		
		Suppose that $c_1(\xi_{n+2}') = a'x_{n+1}' + y'$ for $a' \in \Z$ and $y' \in H^2(B_n)$. Then the bundle $B_{n+2}' \to B_n$ is a fiber bundle with fiber $\Sigma_{a'}$. Let $\overline{x_1'}$ and $\overline{x_2'}$ be the image of $x_{n+1}'$ and $x_{n+2}'$ by the projection $H^2(B_{n+2}') \to H^2(\Sigma_{a'}) \cong H^2(B_{n+2}')/H^2(B_n)$, respectively. 
		Let $\widetilde{\psi} \co H^2(B_{n+2}) \to H^2(B_{n+2}')$ be a homomorphism as $\Z$-modules. Suppose that 
		\begin{equation*}
			\begin{split}
				\widetilde{\psi} (\overline{x_1}) &= \psi_{11}\overline{x_1'} + \psi_{21}\overline{x_2'}+v_1,\\
				\widetilde{\psi} (\overline{x_2}) &= \psi_{12}\overline{x_1'} + \psi_{22}\overline{x_2'}+v_2,
			\end{split}
		\end{equation*}
		here $\psi_{ij} \in \Z$ for $i, j=1, 2$ and $v_1, v_2 \in H^2(B)$. Assume that $\widetilde{\psi}$ extends to an isomorphism of $H^*(B_n)$-algebras. Then $\widetilde{\psi}$ descends to an isomorphism $\psi\co H^*(\Sigma_a) \to H^*(\Sigma_{a'})$ whose representation matrix with respect to $\overline{x_1}, \overline{x_2}$ and $\overline{x_1'}, \overline{x_2'}$ is 
			$\begin{pmatrix}
				\psi_{11} & \psi_{12}\\
				\psi_{21} & \psi_{22}
			\end{pmatrix}	$. In particular, $a$ and $a'$ have the same parity. If the representation matrix of $\psi$ is upper triangular, then $\widetilde{\varphi}$ is induced by a bundle isomorphism $B_{n+2} \to B_{n+2}'$ by Theorem \ref{theo:uppertriangular}. In the sequel of this section, we always assume that the representation matrix of $\psi$ is not upper triangular. 
		Let $\widetilde{\varphi} \co H^*(B_{n+2}') \to H^*(B_{n+2}')$ be the automorphism as an $H^*(B_{n})$-algebra given by $\widetilde{\varphi}(x_{n+1}') = x_{n+1}'$ and $\widetilde{\varphi}(x_{n+2}') = -x_{n+2}'+ c_1(\xi_{n+2}')$. Then $\widetilde{\psi}^{-1} \circ \widetilde{\varphi} \circ \widetilde{\psi} \co H^*(B_{n+2}) \to H^*(B_{n+2})$ is an automorphism as an $H^*(B_n)$-algebra. $\widetilde{\psi}^{-1} \circ \widetilde{\varphi} \circ \widetilde{\psi} \co H^*(B_{n+2}) \to H^*(B_{n+2})$ descends to an isomorphism $H^*(\Sigma_a) \to H^*(\Sigma_a)$ whose representation matrix with respect to $\overline{x_1}, \overline{x_2}$ is $\begin{pmatrix}
				-1 & -a\\
				0 & 1
			\end{pmatrix}$. 	
			
		As before, for a cohomology class $\alpha$ of degree $2$, we denote by $\gamma_k^\alpha$ a line bundle over $B_k$ such that $c_1(\gamma_k^\alpha) = \alpha$. We also denote by $\gamma_{n+1'}^\alpha$ and $\gamma_{n+2'}^\alpha$ line bundles over $B_{n+1}'$ and $B_{n+2}'$, respectively. 
		\begin{enumerate}
			\item Suppose that $a$ and $a'$ are $0$. Since $a$ and $a'$ are $0$, $B_{n+2}$ is isomorphic to the fiber product of $B_{n+1} = P (\underline{\C} \oplus \xi_{n+1}) \to B_n$ and $P(\underline{\C} \oplus \gamma_n^y) \to B_n$ and $B_{n+2}'$ is the fiber product of $B_{n+1}' = P(\underline{\C} \oplus \xi_{n+1}')\to B_n$ and $P(\underline{\C}\oplus \gamma_n^{y'}) \to B_n$. We will show that $P(\underline{\C}\oplus \xi_{n+1}) \cong P(\underline{\C}\oplus \gamma_n^{y'})$ and $P(\underline{\C} \oplus \gamma_n^y) \cong P(\underline{\C} \oplus \xi_{n+1}')$ as bundles over $B_n$. Since the primitive square zero elements in $H^2(\Sigma_0)$ are $\pm \overline{x_1}= \pm \overline{x_1'}$ and $\pm \overline{x_2} = \pm \overline{x_2'}$, and $\psi$ is an isomorphism whose representation matrix is not upper triangular, we have that $\psi(\overline{x_1}) = s_1\overline{x_2'}$ and $\psi(\overline{x_2}) = s_2\overline{x_1'}$ for some $s_1, s_2 = \pm 1$. It follows from $\widetilde{\psi}(x_{n+1})\widetilde{\psi}(x_{n+1}-c_1(\xi_{n+1}))=0$ that  
				\begin{equation*}
					\begin{split}
						0 &= \widetilde{\psi}(x_{n+1})\widetilde{\psi}(x_{n+1}-c_1(\xi_{n+1}))\\
						&= x_{n+2}'(y'+2s_1v_1-s_1c_1(\xi_{n+1}))+ v_1(v_1-c_1(\xi_{n+1})). 
					\end{split}
				\end{equation*}
				Thus we have $y' \pm c_1(\xi_{n+1})$ is even and $y'^2 = c_1(\xi_{n+1})^2$. Therefore
				\begin{equation*}
					\begin{split}
						P(\underline{\C} \oplus \xi_{n+1}) &\cong P(\underline{\C} \oplus \gamma_n^{c_1(\xi_{n+1})}) \\
						&\cong P(\gamma_n^{\frac{1}{2}(y'-c_1(\xi_{n+1}))} \oplus \gamma_n^{\frac{1}{2}(c_1(\xi_{n+1}) + y')}) \\
						& \cong P(\underline{\C} \oplus \gamma_n^{y'})
					\end{split}
				\end{equation*}
				by Lemma \ref{lemm:tensor} and Theorem \ref{theo:decomposable}.
				It follows from $\widetilde{\psi}(x_{n+2})\widetilde{\psi}(x_{n+2}-c_1(\xi_{n+2}))=0$ that  
				\begin{equation*}
					\begin{split}
						0 &= \widetilde{\psi}(x_{n+2})\widetilde{\psi}(x_{n+2}-ax_{n+1}-y)\\
						&= x_{n+1}'(c_1(\xi_{n+1}')-2s_2v_2-s_2y) + v_2(v_2-y).
					\end{split}
				\end{equation*}
				Thus we have $c_1(\xi_{n+1}') \pm y$ is even and $c_1(\xi_{n+1}')^2 =y^2$. Therefore
				\begin{equation*}
					\begin{split}
						P(\underline{\C} \oplus \xi_{n+1}') &\cong P(\underline{\C} \oplus \gamma_n^{c_1(\xi_{n+1}')}) \\
						&\cong P(\gamma_n^{\frac{1}{2}(y-c_1(\xi_{n+1}'))} \oplus \gamma_n^{\frac{1}{2}(c_1(\xi_{n+1}') + y)}) \\
						& \cong P(\underline{\C} \oplus \gamma_n^{y})
					\end{split}
				\end{equation*}
				by Lemma \ref{lemm:tensor} and Theorem \ref{theo:decomposable}. Therefore $B_{n+2}$ and $B_{n+2}'$ are isomorphic as bundles over $B_n$. 
				
			\item Suppose that $a$ and $a'$ are even and one of them is nonzero. 
			If $a$ is nonzero, then it follows from (2-ii) in Section \ref{sec:auto} that $y=-\frac{a}{2}c_1(\xi_{n+1})$. Thus by Lemma \ref{lemm:tensor} and Theorem \ref{theo:decomposable} we have that
			\begin{equation*}
				\begin{split}
					P(\underline{\C} \oplus \xi_{n+2}) & \cong P(\underline{\C} \oplus \gamma_{n+1}^{ax_{n+1}+y})\\
					& \cong P(\underline{\C}\oplus \gamma^{ax_{n+1}-\frac{a}{2}c_1(\xi_{n+1}))})\\
					& \cong P(\gamma_{n+1}^{-\frac{a}{2}x_{n+1}} \oplus \gamma_{n+1}^{\frac{a}{2}(x_{n+1}-c_1(\xi_{n+1}))})\\
					& \cong P(\underline{\C} \oplus \gamma_{n+1}^{-\frac{a}{2}c_1(\xi_{n+1})})\\
					& \cong P(\underline{\C} \oplus \gamma_{n+1}^{y})
				\end{split}
			\end{equation*}
			as bundles over $B_{n+1}$. Therefore we may assume that $a=0$. Using the same argument, we may assume that $a'=0$. It follows from (1) in this section that $B_{n+2}$ and $B_{n+2}'$ are isomorphic as bundles over $B_n$. 
			
			\item Suppose that $a$ and $a'$ are odd. 
			By (3-ii) in Section \ref{sec:auto}, we have that $y=-\frac{a}{2}c_1(\xi_{n+1})$. By the same argument we have that $y'=-\frac{a'}{2}c_1(\xi_{n+1}')$. 
			Since the primitive square zero elements in $H^2(\Sigma_a)$ (respectively, $H^2(\Sigma_{a'})$) are $\pm \overline{x_1}$ (respectively, $\pm \overline{x_1'}$) and $\pm (2\overline{x_2} - a\overline{x_1})$ (respectively, $\pm (2\overline{x_2'} - a'\overline{x_1'})$) and the representation matrix of the isomorphism $\psi \co H^2(\Sigma_a) \to H^2(\Sigma_{a'})$ is not upper triangular, we have that $\psi(\overline{x_1}) = s_1(2\overline{x_2'}-a'\overline{x_1'})$ and $\psi(2\overline{x_2}-a\overline{x_1}) = s_2\overline{x_1'}$ for $s_1, s_2 = \pm 1$. Then $\psi(\overline{x_2}) = as_1\overline{x_2'} + \frac{s_2-s_1aa'}{2}\overline{x_1'}$.	
			It follows from $\widetilde{\psi}(x_{n+1})\widetilde{\psi}(x_{n+1}-c_1(\xi_{n+1}))=0$ that 
			\begin{equation*}
				\begin{split}
					0 & = \widetilde{\psi}(x_{n+1})\widetilde{\psi}(x_{n+1}-c_1(\xi_{n+1})) \\
					&= x_{n+2}'(-2a'c_1(\xi_{n+1}')+4s_1v_1-2s_1c_1(\xi_{n+1}))\\
					& \quad + s_1a'x_{n+1}'(s_1a'c_1(\xi_{n+1}')-2v_1+c_1(\xi_{n+1}))\\
					& \quad + v_1(v_1-c_1(\xi_{n+1})).
				\end{split}
			\end{equation*}
			Thus we have $v_1=\frac{1}{2}(s_1a'c_1(\xi_{n+1}') + c_1(\xi_{n+1}))$ and $a'^2c_1(\xi_{n+1}')^2=c_1(\xi_{n+1})^2$. It follows from $\widetilde{\psi}(x_{n+2})\widetilde{\psi}(x_{n+2}-c_1(\xi_{n+2}))=0$ that 
			\begin{equation*}
				\begin{split}
					0 &=\widetilde{\psi}(x_{n+2})\widetilde{\psi}(x_{n+2}-c_1(\xi_{n+2}))\\
					&=x_{n+1}'(\frac{1-s_1s_2aa'}{4}c_1(\xi_{n+1}')+s_2v_2)\\
					& \quad + v_2(v_2-\frac{s_1aa'}{2}c_1(\xi_{n+1}')).
				\end{split}
			\end{equation*}
			Thus we have $v_2 = \frac{s_1aa'-s_2}{4}c_1(\xi_{n+1}')$ and $(a^2a'^2 -1)c_1(\xi_{n+1}')^2 = 0$. 
			If $a^2a'^2\neq 1$, then $c_1(\xi_{n+1}')^2 = 0$. By considering $\widetilde{\psi}^{-1}$ instead of $\widetilde{\psi}$, we also have that $c_1(\xi_{n+1})^2=0$. By (3-ii) in Section \ref{sec:auto}, we have that $c_1(\xi_{n+1})$ is even. By the same argument we also have that $c_1(\xi_{n+1}')$ is even. 
			Using the same argument as (3) in Section \ref{sec:realizing}, we have that $B_{n+2}$ and $B_{n+2}'$ are trivial bundles over $B_n$. Since the parity of $a$ and $a'$ are the same, we have that $B_{n+2}$ and $B_{n+2}'$ are isomorphic as bundles. 
			
			Suppose that $a^2a'^2=1$. Then $a' = \pm 1$ and we have that $c_1(\xi_{n+1}')^2 = c_1(\xi_{n+1})^2$ and $c_1(\xi_{n+1}') \pm c_1(\xi_{n+1})$ is even. Therefore it follows from Lemma \ref{lemm:tensor} and Theorem \ref{theo:decomposable} that
			\begin{equation*}
				\begin{split}
					P(\underline{\C} \oplus \xi_{n+1}) &\cong P(\underline{\C}\oplus \gamma_n^{c_1(\xi_{n+1})})\\
					&\cong P(\gamma_n^{\frac{1}{2}(-c_1(\xi_{n+1})+c_1(\xi_{n+1}'))} \oplus \gamma_n^{\frac{1}{2}(c_1(\xi_{n+1})+c_1(\xi_{n+1}'))})\\
					&\cong P(\underline{\C}\oplus \gamma_n^{c_1(\xi_{n+1}')})\\
					&\cong P(\underline{\C} \oplus \xi_{n+1}')
				\end{split}
			\end{equation*}
			as bundles over $B_n$. Let $\widetilde{h} \co B_{n+1} \to B_{n+1}'$ be the bundle isomorphism and $\xi_{n+2}''$ the pull-back bundle of $\xi_{n+2}'$ by $\widetilde{h}$. Let $B_{n+2}'' \to B_{n+1}$ be the projectivization $P(\underline{\C} \oplus \xi_{n+2}'') \to B_{n+1}$ and $x_{n+2}''$ the first Chern class of the tautological line bundle of $B_{n+2}$. Let $\widetilde{g} \co B_{n+2}'' \to B_{n+2}'$ be the isomorphism as bundles over $B_{n+1}'$ induced by the pull-back  as bundles over $B_{n+1}$. We have the commutative diagram
				\begin{equation*}
					\begin{tikzcd}[ampersand replacement = \&]
						B_{n+2} = P(\underline{\C}\oplus \xi_{n+2}) \arrow[dr] \& \& \\
						\& B_{n+1} = P(\underline{\C} \oplus \xi_{n+1}) \arrow[dr] \arrow[dd, "\widetilde{h}"] \& \\
						B_{n+2}'' = P(\underline{\C}\oplus \xi_{n+2}'') \arrow[dr]\arrow[dd, "\widetilde{g}"] \arrow[ur] \& \& B_n, \\
						\& B_{n+1}' = P(\underline{\C}\oplus \xi_{n+1}') \arrow[ur] \& \\
						B_{n+2}' = P(\underline{\C}\oplus \xi_{n+2}') \arrow[ur] 
					\end{tikzcd}
				\end{equation*}
				here vertical arrows are isomorphisms as $\C P^1$-bundles and right arrows are projections of $\C P^1$-bundles. 
			The representation matrix of $\widetilde{g}^* \circ \widetilde{\psi} \co  H^*(B_{n+2}) \to H^*(B_{n+2}'')$ is not upper triangular because the one of $\widetilde{g}^*$ is upper triangular but the one of $\widetilde{\psi}$ is not upper triangular. Thus we have that there exists an odd $a''$ such that $c_1(\xi_{n+2}'') = a''x_{n+1} - \frac{a''}{2}c_1(\xi_{n+1})$ by (3-ii) in Section \ref{sec:auto}. If $a'' \neq \pm 1$, then $a^2a''^2 \neq 1$. By the same argument as the case when $a^2a'^2 \neq1$ we have that $B_{n+2}$ and $B_{n+2}''$ are trivial bundles over $B_n$. If $a'' = \pm 1$, then we have that $B_{n+2}$ and $B_{n+2}''$ are isomorphic as bundles over $B_{n+1}$ because $c_1(\xi_{n+2}'') = \pm c_1(\xi_{n+2})$. So $B_{n+2}$ and $B_{n+2}'$ are isomorphic as bundles over $B_n$. 
		\end{enumerate}
\bibliographystyle{alpha}
\bibliography{bibliography.bib}

\end{document}